\documentclass[a4paper,10pt]{amsart}

 \usepackage{graphicx,color}
\usepackage{geometry}
\usepackage{color, graphicx}
\usepackage[all]{xy}
\usepackage{array}

\usepackage{mathrsfs}
\usepackage{bbold}
\usepackage[T1]{fontenc} 
\usepackage{textcomp}
\usepackage{times}
 \usepackage[scaled=0.92]{helvet} 

\usepackage{amssymb}
\usepackage[active]{srcltx}

\newcommand{\Q }{\mathbf{Q} }

\newcommand{\R }{\mathbf{R} }
\newcommand{\PP }{\mathbf{P} }

\renewcommand{\Re}{\mathsf{Re}}
\swapnumbers
\newtheorem{prop}{Proposition}[section]
\newtheorem{lemma}[prop]{Lemma}

\newtheorem{theorem}[prop]{Theorem}

\theoremstyle{definition}

\newtheorem{note}[prop]{Notation}
\newtheorem{definition}[prop]{Definition}
\newtheorem{rem}[prop]{Remark}
\newtheorem{ex}[prop]{Example}
\newtheorem{app}[prop]{Application}
\newtheorem{se}[prop]{}

\usepackage{amsmath,amssymb,amsfonts,amsthm}

\begin{document}

  \date{\today\ (version 1.0)} 
\title[Families of Dirichlet series]{Convergence of families of Dirichlet series}
\author[G.~Cornelissen]{Gunther Cornelissen}
\address{\normalfont Mathematisch Instituut, Universiteit Utrecht, Postbus 80.010, 3508 TA Utrecht, Nederland}
\email{g.cornelissen@uu.nl}
\author[A.~Kontogeorgis]{Aristides Kontogeorgis}
\address{\normalfont Department of Mathematics,  
National and Kapodistrian University of Athens, 
Panepistimioupolis, GR-157 84, Athens, 
Greece } 
\email{kontogar@math.uoa.gr}

\subjclass[2010]{}
\keywords{\normalfont Dirichlet series, zeta function, Riemannian manifold, spectrum, convergence}

  \maketitle

\begin{abstract}
We give some conditions under which (uniform) convergence of a family of 
Dirichlet series to another Dirichlet series implies the convergence of their individual 
coefficients and/or exponents. We give some applications to some spectral zeta functions that arise in 
Riemannian geometry and physics.
\end{abstract}

\section{Introduction} 
Suppose that we have pointwise convergence of a sequence of general Dirichlet series 
\begin{equation} \label{D} D_n(s)=\sum_{\nu \geq 1} a_{n,\nu} e^{ -s \mu_{n,\nu}} \rightarrow D(s)=\sum_{\nu \geq 1} a_\nu e^{ -s \mu_\nu}, \end{equation}
all of which converge absolutely in a common half plane $\Re(s)>\gamma$, 
with $a_*$ complex coefficients, and $\mu_*$ is a \emph{strictly} increasing sequence of real numbers. In this paper, we study what this implies about ``convergence'' of the sequences $(a_{n,\nu})$ and $(\mu_{n,\nu})$. First, we consider the case of classical Dirichlet series, where $a_*=1$ and $\mu_* = \log \lambda_*$, and in Theorem \ref{main}, we prove that (\ref{D}) is equivalent to the $\ell^1$-convergence of $(\lambda^\gamma_{n,\nu})$ to $(\lambda^\gamma_{\nu})$. 
In Theorem \ref{qwe}, we consider the case where $\mu_{n,\nu}=\mu_\nu$ is independent of $n$; 
then, for every $n$, $(a_{n,\nu})$ converges to $(a_n)$.
The most general case is studied in Section \ref{general} from the point of view of the Perron formula for Dirichlet series (where we give a concrete result under some hypotheses), and form the point of view of the Laplace-Stieltjes transform in Section \ref{L}, where we prove the equivalence to Lipschitz convergence of some step functions. 

One application is to Riemannian geometry. If $\{X_n\}_{n=1}^\infty \cup \{X\}$ is a sequence of connected closed smooth Riemannian manifolds such that $d:=\sup \dim X_n$ is finite, convergence of their zeta functions implies convergence of their spectra. 
The application to Riemannian manifolds seems relevant in the light of proposals in physics to use the Laplace spectrum as parameter space in cosmological averaging problems \cite{Seriu}, and as dynamical variables in classical gravity, inspired by particle models coupled to gravity in noncommutative geometry \cite{LandiRovelli}, \cite{Connes}. We discuss this briefly in Section \ref{phys}, where we show how to use spectral zeta functions to introduce a metric on spaces of Riemannian manifolds up to isospectrality.  By the analogy between manifolds and number fields, we similarly introduce a metric on the space of number fields up to arithmetic equivalence, using  topology deduced from convergence of their Dedekind zeta functions.

\section{Convergence of Dirichlet series} 

\begin{note} Let us once and for all introduce the following convenient notation: if $s \in \mathbf{C}$ and $\Lambda=(\lambda_\nu)_{\nu=1}^\infty$ is a sequence of positive real numbers, we  denote by $\Lambda^s$ the sequence $(\lambda_\nu^s)_{\nu=1}^\infty$. 
\end{note} 

\begin{theorem} \label{main}
Suppose that 
\[
 D_n(s)=\sum_{\nu \geq 1}  \lambda_{n,\nu}^{-s} \mbox{ and } D(s):=\sum_{\nu \geq 1 } \lambda_\nu^{-s}
\]
is a family of (generalized) Dirichlet series for $n=\emptyset, 1,2,\dots$, where, for each $n$, $\Lambda_n:=(\lambda_{n,\nu})_{\nu=1}^\infty$ 
forms a sequence of increasing positive real numbers with finite multiplicities. 
Assume that all series $D_n(s)$ are convergent in a common right half plane $\Re(s)>\gamma>0$. Then the following are equivalent: 
\begin{enumerate}
\item[\textup{(i)}] As $n \rightarrow +\infty$, the functions $D_n(s)$ converge to $D(s)$, pointwise in $s$ with $\Re(s)>\gamma$;
\item[\textup{(ii)}]  For every fixed $\nu$, any bounded subsequence of $\{\lambda_{n,\nu}\}_{n=1}^\infty$ converges to the same element $\lambda \in \Lambda$, and 
\[
 \#\{ (\lambda_{n,\nu})_{n=1}^\infty: \lim_{n \rightarrow \infty}\lambda_{n,\nu}= \lambda \}=
\#\{ \lambda_\nu: \lambda_\nu=\lambda\}.
\]
\item[\textup{(iii)}]  $\Lambda^{-\gamma}_n$ converges to $\Lambda^{-\gamma}$ in $\ell^1$.  
\end{enumerate}
\end{theorem}

\begin{rem} The assumption that all series $D_n(s)$ are convergent in a common right half plane $\Re(s)>\gamma$ is a minimal necessary assumption, since if this is not the case, the questions we ask are void.
\end{rem} 

\begin{rem}
The series $D_n$ (when divergent at $s=0$), converges for $\Re(s)>\gamma_n$, where (see Chapter 1, Section 6 of \cite{Hardy}): \begin{equation} \label{hardy1}
 \gamma_n=\limsup_{\nu\rightarrow \infty}\frac{\log \nu }{\log \lambda_{n,\nu}}.
\end{equation}
The hypothesis says that $\gamma:=\sup\gamma_n$ is finite. 
\end{rem}

\begin{rem}
The spectral zeta function $\zeta_X$ of a closed smooth Riemannian manifold converges absolutely for $\Re(s)>d/2$, where $d$ is the dimension of $X$ \cite{Rosenberg}. In the situation of Theorem \ref{Riem}, the assumption of a common half-plane of convergence hence follows from the fact that we assume that $\sup \dim(X_n)$ is finite. Hence Theorem \ref{Riem} follows from Theorem \ref{main}. 
\end{rem}

\begin{proof}[Proof of Theorem \ref{main}]
Since we later want to interchange some limits, we will first prove: 

\begin{lemma}\label{U} The sequences 
\[
 S_{n,N}(s):=  \sum_{\nu=1}^N \frac{1}{\lambda_{n,\nu}^s}
\]
converge to $D_n(s)$ for $N \rightarrow + \infty$ \emph{uniformly in $n$}.
\end{lemma}

\begin{proof}[Proof of Lemma \ref{U}] The uniform convergence means that  
\[
 \forall \epsilon>0, \exists N_0 \geq 1, \forall N>N_0, \;\;\; 
 \left|
  \sum_{\nu=1}^N \frac{1}{\lambda_{n,\nu}^s}-D_n(s)
 \right|\leq \epsilon
\]
where $N_0$ \emph{does not depend on $n$}.
Now notice that by Equation  (\ref{hardy1}) we have 
\[
 \gamma_n= \lim \sup_{\nu \rightarrow \infty }\frac{ \log \nu}{\log \lambda_{n,\nu}}
\]
and the series $D_n(s)$ converge for $\mathrm{Im}(s) \geq \gamma_n$. 
The assumption that there is a common half plane of convergence for all $D_n$ means that the sequence $\gamma_n$ is bounded
by $\gamma=\sup \gamma_n$. 
The uniform convergence of $S_{n,N}(s)$ follows by repeating the argument found in
the proof on page 7 of \cite{Hardy}. The proof there uses $\gamma_n$ for each $n$, but one may as well use (the same) $\gamma$ for all $n$. 
\end{proof} 

Next, we will show that unbounded subsequences do not contribute to the limit. For this, suppose that, for some fixed $\kappa$, 
$(\lambda_{n_k,\kappa})_{k\in \mathbf{N}}$ is an unbouded subsequence, with $$\lim_{k\rightarrow \infty} \lambda_{n_k,\kappa}=\infty$$ and, by enlarging the subsequence if necessary,  such that the sequence
\[
 (\lambda_{n,\kappa})_{n\in (\mathbf{N}- \{n_k: k \in \mathbf{N} \}) }  
\]
is either bounded or the empty set. Now since $\lambda_{n_k,\kappa} \leq \lambda_{n_k,\mu}$ for $\mu \geq \kappa$, 
all sequences $(\lambda_{n_k,\mu})_{k\in \mathbf{N}}$  for $\mu \geq \kappa$ tend to infinity as 
well. Observe now that the series 
\[
 D^{\geq \kappa}_{n_k}(s):=\sum_{\nu=\kappa}^\infty \frac{1}{\lambda_{n_k,\nu}^s} 
\]
tends to the zero function as $n_k$ tends to infinity: 
\[
 \lim_{n_k\rightarrow \infty} D^{\geq \kappa}_{n_k}(s)= \sum_{\nu=\kappa}^\infty 
 \lim_{n_k\rightarrow \infty} \frac{1}{\lambda_{n_k,\nu}^s}=0,
\]
for $s$ real positive (hence for all $s$ by analytic continuation). 
In the above equation we were allowed to interchange the order
of the limits (in $n_k$ and the summation variable of $D_{n_k}$) since the series converge uniformly in $n_k$. 
Since we assume that $D_n$ is a convergent sequence of functions, it has the same limit as its the subsequence $D_{n_k}$. 

Since we have now proven that unbounded subsequences do not contribute to the limit, we can assume that $\lambda_{n,\nu}$ is  bounded in $n$, for all $\nu$, i.e., 
\[
 \forall n\in \mathbf{N}\;\; \lambda_{n,\nu} \leq c_\nu.
\]
Then we can select a subsequence so that for all $\nu$, the limit
$$\lim\limits_{k\rightarrow \infty} \lambda_{n_k,\nu}=\ell_\nu$$ exists. 
Not to overload notation, we will momentarily relabel the convergent subsequence $\lambda_{n_k,\nu}$ as $\lambda_{n,\nu}$. 
In particular,  $\lambda_{n,1}$ converges to $\ell_1$.
We will prove that $\lambda_{n,1}$ converges to $\lambda_1$. 
Let us rewrite 
\[
 D_n(s)=\frac{1}{\lambda_{n,1}^s}\left(\sum_{\nu=1}^\infty  
\left(\frac{\lambda_{n,1}}{\lambda_{n,\nu}}\right)^s \right)
\]
and 
\[
 D(s)=\frac{1}{\lambda_{1}^s}\left(\sum_{\nu=1}^\infty  
\left(\frac{\lambda_{1}}{\lambda_{\nu}}\right)^s \right).
\]
We now assume that \emph{$s$ is an integer $s>\gamma$}. 
Since $D_n(s)\rightarrow D(s)$ we have that
\begin{equation} \label{11}
 \lim_{n\rightarrow \infty}  
\left(
\frac{\lambda_1}{\lambda_{n,1}}
\right)^s=
\frac{
 \sum\limits_{\nu=1}^\infty  
\left(\frac{\lambda_{1}}{\lambda_{\nu}}\right)^s
}
{
   \lim_{n\rightarrow \infty} \sum\limits_{\nu=1}^\infty 
\left(\frac{\lambda_{n,1}}{\lambda_{n,\nu}}\right)^s
}
\leq \sum_{\nu=1}^\infty  
\left(\frac{\lambda_{1}}{\lambda_{\nu}}\right)^s
\end{equation}
For the last inequality, we have used the fact that $\lambda_{n,\nu}>0$ and 
that the denominator is $\geq 1$.

Set 
$$\ell:=\lim_{n\rightarrow \infty} \frac{\lambda_1}{\lambda_{n,1}}.$$
We now consider the limit as $s \rightarrow \infty$ (along the integers) in Equation (\ref{11}), to find
\[
  \lim_{s\rightarrow \infty} \ell^s \leq 
\#\{\lambda_i =\lambda_1 \}.
\]

Now
\[
 \lim_{s \rightarrow \infty} \ell^s=
\left\{
\begin{array}{ll}
 1 & \mbox{ if } \ell=1 \\
 0 & \mbox{ if } \ell<1 \\
 \infty & \mbox{ if } \ell >1
\end{array}
  \right..
\]
Hence we find $\ell \leq 1$. 

We also have the inequality 
\begin{equation} \label{12} 
\ell^s= \lim_{n\rightarrow \infty}  
\left(
\frac{\lambda_1}{\lambda_{n,1}}
\right)^s=
\frac{
 \sum\limits_{\nu=1}^\infty 
\left(\frac{\lambda_{1}}{\lambda_{\nu}}\right)^s
}
{
   \lim_{n\rightarrow \infty} \sum\limits_{\nu=1}^\infty 
\left(\frac{\lambda_{n,1}}{\lambda_{n,\nu}}\right)^s
}
\geq 
\frac{1}
{
 \sum\limits_{\nu=1}^\infty 
 \left(\frac{\ell_1}{\ell_\nu }\right)^s
}.
\end{equation}
In the inequality, we have used that we can interchange limit and summation in the denominator, by uniform convergence. 

By taking the limit $s\rightarrow \infty$ (along the integers) we arrive at
\[
 \lim_{s\rightarrow \infty} \ell^s \geq \frac{1}{\#\{ \ell_n=\ell_1 \}} >0.
\]
We conclude from all the above that $\ell=1$, and hence that 
\[
 1= \frac{\#\{\lambda_i =\lambda_1 \}}{\#\{ \ell_n=\ell_1 \}}.
\]
Now recall that we have relabelled before, so that we have actually shown that every convergent  subsequence $(\lambda_{n_k,1})_{k\in \mathbf{N}}$ of $(\lambda_{n,1})_{n\in \mathbf{N}}$  tend to some limit, and since $\ell=1$ all these subsequences converge to the \emph{same} limit 
 $\lambda_{1}$. Therefore $(\lambda_n)_{n\in \mathbf{N}}$ itself is convergent to $\lambda_1$. We conclude that in general (viz., before erasing all unbounded subsequences), that every bounded subsequence of $(\lambda_n)_{n\in \mathbf{N}}$ converges to $\lambda_1$. 

We now use an inductive argument to treat the general term. Namely, consider the Dirichlet  series
$$D^{\geq 2}_n(s):=D_n(s) -\lambda_{n,1}^{-s}$$ which (by what we have proven) converges to $$D^{\geq 2}(s):=D(s)-\lambda_1^{-s}.$$
These are still sequences of Dirichlet series of the same form, but with first eigenvalues $\lambda_{n,2}$ and $\lambda_2$. We can repeat the argument with this series, to conclude $\lambda_{n,2} \rightarrow \lambda_2$, etc. 

This finishes the proof that (i) implies (ii). 

Since we assume that $\Re(s)>\gamma$ is a common half plane of convergence of all series $D_n$ ($n=\emptyset,1,2,\dots$), the sums 
$\sum_{\nu=1}^\infty {\lambda_{n,\nu}^{-\gamma}}$ converge, and hence the sequences $\Lambda_n^{-\gamma}$ ($n=\emptyset,1,2,\dots$) belong to the Banach space $\ell^1$.
We will now prove that  $\Lambda^{-\gamma}_n\rightarrow \Lambda^{-\gamma}$ as elements of $\ell^1$.

In order to do so we have to prove that for every $\epsilon>0$ there is an $n_0 \in \mathbf{N}$
such that $n>n_0$ implies 
\[
  \sum_{\nu=1}^\infty \left|
 \frac{1}{\lambda_{n,\nu}^{\gamma}} - \frac{1}{\lambda_\nu^{\gamma}}
 \right| \leq \epsilon
\]
It is known that if $(a_\nu)$ is a sequence of positive real numbers so that $\sum_{\nu=1}^\infty a_\nu$ converges,
then all of its ``tails'' tend to zero: $$\lim_{N\rightarrow \infty} \sum_{\nu=N}^\infty a_\nu=0.$$  
So given an $\epsilon> 0$ there is an $n_0$, 
which does not depend on $n$ (using the same $\gamma$ for all $n$),
such that for $N\geq n_0$
\[
 \sum_{\nu=N}^\infty \left|
 \frac{1}{\lambda_{n,\nu}^{\gamma}}\right| + \sum_{\nu=N}^\infty
\left|\frac{1}{\lambda_\nu^{\gamma}}
 \right|  <\epsilon/2.
\]
Therefore, 
\begin{eqnarray}
  \sum_{\nu=1}^\infty \left|
 \frac{1}{\lambda_{n,\nu}^{\gamma}} - \frac{1}{\lambda_\nu^{\gamma}}
 \right| &\leq&  
 \sum_{\nu=1}^N \left|
 \frac{1}{\lambda_{n,\nu}^{\gamma}} - \frac{1}{\lambda_\nu^{\gamma}}
 \right| +
 \sum_{\nu=N}^\infty \left|
 \frac{1}{\lambda_{n,\nu}^{\gamma}}\right| + \sum_{\nu=N}^\infty
\left|\frac{1}{\lambda_\nu^{\gamma}}
 \right|   \nonumber \\
 &\leq &
 \sum_{\nu=1}^N 
 \frac{|\lambda_\nu^{\gamma} -\lambda_{n,\nu}^{\gamma}|}
 {\lambda_{n,\nu}^{\gamma}\lambda_\nu^{\gamma}} + \frac{\epsilon}{2} \nonumber \\
 &\leq& 
  \sum_{\nu=1}^N 
 \frac{|\lambda_\nu^{\gamma} -\lambda_{n,\nu}^{\gamma}|}
 {C} + \frac{\epsilon}{2}, \label{con-bound}
\end{eqnarray}
where 
\[
0\neq C=\inf_{1\leq \nu \leq N} 
(\lambda_{1,\nu}^{\gamma}\lambda_\nu^{\gamma}) \leq 
(\lambda_{n,\nu}^{\gamma}\lambda_\nu^{\gamma}).
\]
Now the finite number ($\nu=1,\ldots,N$) of sequences 
$(\lambda_{n,\nu}^{\gamma})_{n\in \mathbf{N}}$
can be made to uniformly converge to $\lambda_\nu^{\gamma}$, 
that is, for every $\epsilon>0$ there is an $n_1$ such that 
$n>n_1$ implies 

$$|\lambda_\nu^{\gamma} -\lambda_{n,\nu}^{\gamma}| \leq \frac{\epsilon C}{2N}$$
and inequality  (\ref{con-bound}) gives us the desired result
for all $n\geq \max\{n_0,n_1\}$.

This proves that (ii) implies (iii). Finally, if $\Lambda_n^{-\gamma}$ converges to $\Lambda^{-\gamma}$, we have for every $s \in \mathbf{C}$ with $\Re(s)>\gamma$ that $\Lambda_n^{-s}$ converges to $\Lambda^{-s}$, and it follows easily that $D_n(s)$ converges to $D(s)$, pointwise in $s$. This proves that (iii) implies (i) and finishes the proof of the theorem. 
\end{proof}

\begin{app} Suppose $X=(X,g_X)$ and $Y=(Y,g_Y)$ are two isospectral connected smooth closed Riemannian manifolds, i.e., suppose their Laplace-Beltrami operators $\Delta_X$ and $\Delta_Y$ have the same spectrum with multiplicities $\Lambda_X=\Lambda_Y$ \cite{Rosenberg}. The spectrum $\Lambda_X$ is considered as a sequence $(\lambda_\nu)_{\nu=1}^\infty$ with $0 \leq \lambda_1 \leq \lambda_2 \leq \dots$, with finite repetitions. 

The identity theorem for Dirichlet series \cite{Hardy} shows that such isospectrality can also be described as the manifolds having the same \emph{zeta function} $\zeta_X = \zeta_Y$, where
$$ \zeta_{X}(s) := \mathrm{tr}(\Delta_X^{-s}) = \sum_{0 \neq \lambda \in \Lambda_X} \frac{1}{\lambda^s}, $$
since connectedness implies that the zero eigenvalue has multiplicity one. 
In this context, Theorem \ref{main} says the following, which is a ``convergent'' version of the identity theorem: 
\begin{prop} \label{Riem} Suppose $\{X_n\}_{n=1}^\infty$ is a sequence of connected closed smooth Riemannian manifolds such that $d:=\sup \dim X_n$ is finite, and suppose that $X$ is another closed smooth Riemannian manifold. Then the following statements are equivalent
\begin{enumerate}
\item[\textup{(i)}] For $\Re(s)>d/2$, the functions $\zeta_{X_n}(s)$ converge pointwise to $\zeta_X$;
\item[\textup{(ii)}] For some $\gamma \in \mathbf{C}$ with $Re(\gamma)>d/2$, the sequence of eigenvalues $\Lambda^{-\gamma}_{X_n}$ converges to $\Lambda_X^{-\gamma}$ in $\ell^1$.  
\end{enumerate}
\end{prop}
\begin{rem} If the manifolds are closed and smooth and \emph{of odd dimension}, but possibly disconnected, the equality $\zeta_X = \zeta_Y$ implies that also the multiplicity of the zero eigenvalue is equal for $X$ and $Y$, namely, it is minus the value at $0$ of the analytic continuation of $\zeta_X$ (\cite{Rosenberg}, 5.2). 
\end{rem}
\end{app} 

\begin{ex} The circle of radius $r$ has $\Delta=-r^2 \partial^2_\theta$ (with $\theta\in [0,2\pi)$ the angle coordinate), spectrum $\lambda_{r,\nu} = r^{-2} \lceil \nu/2 \rceil^2$ and zeta function $\zeta_r(s)=r^{2s} \zeta(2s)$, where $\zeta$ is the Riemann zeta funtion. For varying $r \rightarrow r_0$, the convergence in the theorem happens for $\gamma>1/2$.  
\end{ex}

\begin{rem}
Already in the case of families of Riemannian manifolds, it can happen that $(\lambda_{n,\kappa})$ has unbounded subsequences for some fixed $\kappa$; for example, a family of circles whose radius tends to zero. However, for fixed $\kappa$, we have bounds on the eigenvalues of the form (\cite{BBG})
\[
 C_1 \sqrt[d]{\kappa^2}  \leq  \lambda_{n,\kappa} \leq 
\frac{C_2}{\mathrm{vol}(X_n)} \sqrt[d]{\kappa^2},
\]
where the constants $C_i$ depend on the dimension $d$, the diameter $D$, and a lower bound $R$ on the Ricci curvature of the manifolds under consideration. This implies that (at least if we fix the data $d,D$ and $R$, so we are in the Gromov precompact moduli space \cite{Gromov}) in unbounded subsequences, the volume should shrink to zero. 
\end{rem}

\section{Applications: metric theories derived from Dirichlet series}  \label{phys} 

\subsection*{Distances in cosmology} In connection with the averaging problem in cosmology and the question of topology change under evolution of the universe, Seriu \cite{Seriu} proposes to use eigenvalues of the Laplace-Beltrami operator on spatial sections of a cosmological model to construct a metric on the space of such Riemannian manifolds up to some notion of ``large scale isospectrality''. More precisely, he raises two objections against the use of the plain difference of spectra as a measure: large energy contributions (corresponding to small scale geometry) should carry a lower weight, and the dominant weight should be put on the small spectrum (corresponding to large scale geometry); therefore, he introduces a cut-off $N$ and only compares the first $N$ eigenvalues. Secondly, the eigenvalue difference is not a dimensionless quantity, and because of this, he suggests comparing quotients of spectra. However, as $N \rightarrow +\infty$, his distance diverges.

From the above theorem, it also appears natural not to use a cut-off function, but rather use a distance \emph{between the zeta functions} (which, like partition functions, give more weight to low energy in their region of convergence), considered as complex functions; here, one may use classical notions of distance between complex functions \cite{Conway} used in the study of limits of holomorphic or meromorphic functions.  Also, the \emph{quotient} of two zeta functions is a dimensionless function. Actually, a distance between Riemannian manifolds \emph{up to isometry} was constructed by the first author and de Jong, who have furthermore given a spectral characterization of when a diffeomorphism of closed smooth Riemannian manifolds is an isometry, in terms of equality of more general zeta functions under pullback by the map \cite{CdJ}. Also this distance is based on the dimensionless object of quotients of zeta functions. 

In conclusion, we propose the following function as a distance on suitable spaces of Riemannian geometries up to isospectrality: 

\begin{prop} Let $\mathcal{M}$ denote a space of Riemannian manifolds up to isospectrality,  with $$\sup \{ \dim(X) \colon X \in \mathcal{M} \} <2 \gamma$$ finite. Then for any $X_1, X_2 \in \mathcal{M}$, the function
$$ d(X_1,X_2):= \sup_{\gamma < s < \gamma+1} \left| \log \left| \frac{\zeta_{X_1}(s)}{\zeta_{X_2}(s)} \right| \right| $$
where $\Re(s)>\gamma$ is common plane of convergence for the spectral zeta functions of $X_1$ and $X_2$, defines a metric on $\mathcal{M}$.
\end{prop}

\begin{proof} The function $d$ is positive, and if $d(X_1,X_2)=0$, then $|\zeta_{X_1}(s)|=|\zeta_{X_2}(s)|$ for all $s$ in the interval $]\gamma,\gamma+1[$. Since this set has accumulation points, and since the zeta function is positive real for such values of $s$, we find that $\zeta_{X_1} = \zeta_{X_2}$ as complex functions. Hence the main theorem (or the identity theorem for Dirichlet series) implies that $X_1$ and $X_2$ are isospectral. The function $d$ is symmetric, since $\left| \log(x^{-1}) \right| = \left| \log(x) \right|$. Finally, the triangle inequality follows from 
$$ \frac{\zeta_{X_1}(s)}{\zeta_{X_3}(s)} = \frac{\zeta_{X_1}(s)}{\zeta_{X_2}(s)} \cdot \frac{\zeta_{X_2}(s)}{\zeta_{X_3}(s)} $$ 
and the usual properties of the absolute value. 
\end{proof}

This is a distance that weighs correctly the energy contributions, but does not depend on a cut-off, nor diverges if a cut-off tends to infinity. Also, convergence in the topology defined by this distance can be easily understood from Theorem \ref{main}. 

\begin{rem}
Taking the supremum over $\gamma < s < \gamma+1$ is quite random, any set with an accumulation point and avoiding the poles of the zeta functions will do. Also, the distance $d$ can be replaced by $d/(1+d)$ to have it take values in the unit interval.  The exact numerical values of the metric are not so relevant, but rather, their interrelation and the topology and uniformity that they induce. 
\end{rem}

\begin{ex} If $S_r$ denotes a circle of radius $r$, then 
$$ d(S_{r_1}, S_{r_2}) = 4 \left| \log(r_1/r_2) \right|. $$ This example shows that the distance can be non-differentiable in the parameter space of a family. 
\end{ex}


%
%
%
\begin{ex}
Let us compute the spectral distance between a sphere $S$ and a real projective space $\R\PP^2$ with the same volume $4 \pi$. The zeta functions are 
$$ \zeta_{S}= \sum_{\nu=1}^\infty \frac{2 \nu +1}{\nu^s (\nu+1)^s} \mbox{ and }   \zeta_{\R\PP^2}= \sum_{\nu=1}^\infty \frac{4 \nu +1}{\nu^s (2 \nu+1)^s}. 
$$
A numerical experiment suggests that the maximum in the distance formula is attained at $s=2$, and there we get
$$ d(S,\R\PP^2) = \log (4-\pi^2/3) \approx 0.342. $$
\end{ex}

\subsection*{Eigenvalues as dynamical variables} Gravity coupled to matter can be given a spectral description using the framework of noncommutative geometry \cite{Connes}. Even by ignoring the matter part, one arives at an interesting description of classical gravity (general relativity) in terms of spectral data. These spectra form a diffeomorphism invariant set of coordinates on the space of manifolds, up to isospectrality. Diffeomorphism invariant coordinates are an important prerequisite for certain programmes to quantize gravity. In this way, spectra were used as dynamical variables for classical gravity by Landi and Rovelli \cite{LandiRovelli}. Our theorem shows that convergence in these spacetime variables is the same as convergence of classical Dirichlet series in complex analysis.

\subsection*{A distance on Number Fields}

Consider the Dedekind-zeta function  for a number field $K$ as $\zeta_K(s):=\sum N(I)^{-s}$, where the sum runs over all non-zero ideals $I$ of  $\mathcal{O}_K$, the ring of integers of $K$, and $N(I)$ is the norm of the ideal $I$. For any fixed $a>0$, we can define a distance 
$$ d(K_1,K_2) :=  \sup_{1 < s < 1+a} \left| \log \left| \frac{\zeta_{K_1}(s)}{\zeta_{K_2}(s)} \right| \right|. $$
For a prime number $p$ and a positive integer $f$, let $I_K(p,f)$ denote the number of ideals of $\mathcal{O}_K$ with norm $p^f$. Since $\zeta_K(s)$ admits an Euler product  
\[
\zeta_K(s)=\prod_{P \in \mathrm{Spec}\mathcal{O}_K} \frac{1}{1-N(P)^{-s}},
\]
we find 
$$
d(K_1,K_2)  = \sup_{1 < s < 1+a} \left| \sum_{p,f} \left(I_{K_1}(p,f)-I_{K_2}(p,f)\right) \log \left(1-p^{-fs}\right) \right|.
$$

\begin{ex}
Let us estimate the distance between the field of rational numbers $\Q$ and a real quadratic field $\Q(\sqrt{D})$ for $D>0$. Let $I$ denote the primes inert in $\Q(\sqrt{D})$, $S$ the set of split primes and $R$ the set of ramified primes, i.e., the divisors of the squarefree part of $D$. 
Using the Euler products, we find 
\begin{align*} d(\Q,\Q(\sqrt{D}) &=  \sup_{1 < s < 1+a} \left| \log \left| L(\chi_D,s)
     \right| \right|
\\ & = \sup_{1 < s < 1+a} \left| \log \left| 
\frac{ 
\prod\limits_{p \in S} \left(1-p^{-s}\right)^{-2}
\cdot \prod\limits_{p \in I} (1-p^{-2s})^{-1}
\cdot \prod\limits_{p \in R}
(1-p^{-s})^{-1}
}
{  \prod\limits_{p} (1-p^{-s})^{-1}   } 
     \right| \right|. 
     \end{align*}

As $D \rightarrow 1$, the distance goes to zero. 
If $D_i$ denotes the product of the first $i$ prime numbers, then $R$ increases to the set of all primes, so then $\lim\limits_{i \rightarrow + \infty} d(\Q,\Q(\sqrt{D_i})) =0$, too. We verified numerically in SAGE that $L(\chi_{D_i},1) \rightarrow 1$ as $i \rightarrow +\infty$; The figure is a plot of $\left| L(\chi_{D_i},1) \right|$ as a function of $i$. 
\begin{figure}[h] 
\includegraphics[width=8cm]{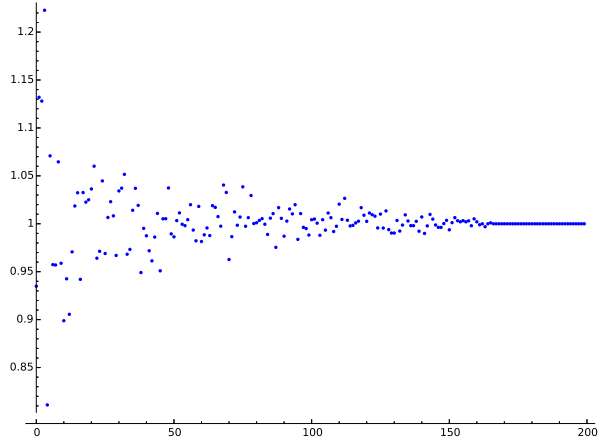}
\caption{$\left| L(\chi_{D_i},1) \right|$ a a function of $i$}
\end{figure} 

\end{ex}

\section{Series with general coefficients}

In this section, we study what happens if we have pointwise convergence of general Dirichlet series 
$$ D_n(s)=\sum_{\nu \geq 1} a_{n,\nu} e^{ -s \mu_{n,\nu}} \rightarrow D(s)=\sum_{\nu \geq 1} a_\nu e^{ -s \mu_\nu}, $$
all of which converge absolutely in a common half plane $\Re(s)>\gamma$, 
with $a_*$ complex coefficients, and $\mu_*$ is a \emph{strictly} increasing sequence of real numbers. The previous case occurs when $\{\mu_*\} = \{\log \lambda_*\}$ and $a_*$ counts the multiplicities in $(\lambda_*)$. In this paper, we will not discuss subtleties that arise from such series that have a different region of convergence and absolute convergence. 
We start by discussing two special cases. 

\begin{se}[Taylor series] The first is when $\mu_{n,\nu}=\log \nu$ for all $n=\emptyset,1,2,\dots$. In this case, we set $z=e^{-s}$ and we get a (pointwise) convergence of Taylor series 
$$ D_n(z) = \sum_{\nu \geq 0} a_{n,\nu} z^\nu \rightarrow D(z) = \sum_{\nu \geq 0} a_{\nu} z^\nu. $$ In this case, the individual series $D_*$ converge in $\Omega:=\{z>e^{-\gamma}\}$ to a holomorphic function (by assumption). Evaluation at zero gives $\lim_{n \rightarrow + \infty} a_{n,0} = a_0$, and we can proceed by induction to conclude that $$\lim_{n \rightarrow + \infty} a_{n,\nu} = a_\nu$$ for all $\nu$. 

Alternatively, one can use the representation of the coefficients by a complex contour integral to deduce the result ``in a more complicated way''. Namely, fix $\epsilon>0$, and let $n_0$ satisfy that $|D_n(z)-D(z)|<\epsilon$ for $n>n_0$, uniformly in $z \in K \subset \Omega$, where $K$ is a compact set. For a contour $C \subset K$ around $z=0$ (independent of $n$), we have 
$$ |a_{n,\nu} - a_\nu| \leq \frac{1}{2 \pi}  \int_{C} |D_n(z)-D(z)| |z|^{-n-1} dz \leq \varepsilon. $$ 
 \end{se} 
 
\begin{se}[Constant exponents] The reason for providing this second proof is that it leads us to the next special case, in which we use the analogue of the integral representation for the coefficients for general Dirichlet series, also called Perron's formula. This formula gives a representation of the terms of a general Dirichlet series by integration over a vertical line in the complex plane, and since this integration domain, unlike the contour in the Taylor series proof, is not compact, we will need to work more to establish the result (or assume uniform convergence on an entire half-line, which seems too strong an assumption).  This second special case occurs if $\mu_{n,\nu}$ is constant in $n$. Then we have the following result: 
\end{se} 

\begin{theorem} \label{qwe}
 Assume that $D_n$ ($n=\emptyset,1,2,\dots$) is a set of Dirichlet series that converge absolutely in   a
common half plane $\Re(s)>\gamma$, and such that $D_n(s) \rightarrow D(s)$ converges pointwise there. Assume that $\mu_{n,\nu}=\mu_\nu$ is independent of $n$.  
Then for every $n$, we have  $$\lim_{n \rightarrow + \infty} a_{n,\nu} = a_\nu;$$
actually, for $\sigma_1>\gamma $, we have a convergence of sequences
$$ (a_{n,\nu} e^{-\sigma_1 \mu_\nu})_{\nu=1}^\infty \rightarrow (a_\nu e^{-\sigma_1 \mu_\nu})_{\nu=1}^\infty \mbox{ in } \ell^{\infty}. $$   
\end{theorem} 

\begin{proof}
 Consider the difference
\[
B_n(s):= D_n(s)-D(s)=\sum_{\nu \geq 1} {b_{n,\nu}}{e^{-s \mu_\nu}},
\]
where $$b_{n,\nu}:=a_{n,\nu}-a_{\nu}.$$ 


Now according to Theorem I.3.1 in \cite{mandel} we have the following integral representation for every $n$ and every fixed $\nu$: 
\begin{eqnarray*}
 \left|
b_{n,\nu}e^{- \sigma_1 \mu_\nu}
\right|
 &= &
\left|
\lim_{T\rightarrow \infty} \frac{1}{T} \int_{0}^T B_n(\sigma_1 + i t) e^{\mu_\nu i t} dt 
\right| \\
 &\leq &
\lim_{T\rightarrow \infty} \frac{1}{T} \int_{0}^T \left| B_n(\sigma_1 + i t) \right|   dt
\end{eqnarray*}
We could finish the proof here by assuming that $D_n$ converges uniformly to $D$ on the entire line $\Re(s)=\sigma_1$. However, we can avoid this  (strong) hypothesis by proving the following lemma: 
\begin{lemma} \label{bound-lemma}
 For every $\epsilon>0$ there is a $t_0 \in \mathbf{R}$ such that for $t \in \R$ and all $n$, 
 \[
  |B_n(\sigma_1 + it)| \leq  \epsilon + |B_n(\sigma_1+it_0)|.
 \]
\end{lemma}
We are then finished with the proof of Theorem \ref{qwe}, since now, given any $\epsilon>0$, the pointwise convergence at $t_0$ implies that there exists $n_0$ such that for all $n>n_0$, 
 \[
  |B_n(\sigma_1 + it)| \leq  \epsilon + |B_n(t_0)| \leq 2 \epsilon
 \]
and then the above inequality becomes  
$$ \left|
b_{n,\nu}e^{- \sigma_1 \mu_\nu}
\right| \leq 2 \epsilon, $$
Since $\sigma_1$ and $\nu$ are fixed, $e^{- \sigma_1 \mu_\nu}$ is a non-zero constant, and this proves that $b_{n,\nu} \rightarrow 0$ as $n \rightarrow +\infty$. Since the $\epsilon$-bound holds uniformly in $\nu$, we do find the $\ell^\infty$ convergence as stated. 

\begin{proof}[Proof of Lemma \ref{bound-lemma}] Since the series $D_n$ are absolutely convergent on a common half plane, 
their sequences of tails tend  to zero uniformly in $n$, that is, for every 
$\epsilon>0$ there is an $N$ that is independent of $n$ such that 
\[
 \sum_{\nu=N+1}^\infty | a_{n,\nu} e^{-s \mu_\nu}|+ \sum_{\nu=N+1}^\infty 
|a_\nu e^{-s \mu_\nu}| < \epsilon.
\]
Hence 
\begin{equation} |B_n(s)| 
\leq \left| B_n^{\leq N}(s) 
+D_n^{> N}(s)  -D^{> N}(s)  \right| 
\leq \left|  B_n^{\leq N}(s) \label{bb1}
 \right| + \epsilon
\end{equation} 
We will now estimate the sum of the first $N$ terms on a vertical line $\Re(s)=\sigma_1$. 
Consider the function $f:\mathbf{R} \rightarrow \left(S^1\right)^{N}$ given by
\[
 t\mapsto (e^{i t \mu_{1}}, \ldots, e^{i t \mu_{N}})
\]
and the function 
\[
 F:\left(S^1\right)^{N} \rightarrow \mathbf{C}
\]
sending
\[
 (P_1,\ldots, P_N) \mapsto 
\sum_{\nu=1}^N 
 e^{-\sigma_1 \mu_\nu} P_\nu.
\]
The function $F$ is continuous on a compact set therefore it attains a 
maximal value $M$ at a point $A^0:=(P_1^0,\ldots, P_N^0)$.
 
\begin{lemma}
There exists $t_0 \in \mathbf{R}$ such that all the numbers $\{t_0 \mu_\nu\}_{\nu=1}^\infty$
are irrational.
\end{lemma}
\begin{proof}
 The set of multiples $\{ b \rho\}$ of a given real number $\rho \in \mathbf{R}$ such that $b\rho \in \mathbf{Q}$ is
 just $\frac{1}{\rho} \mathbf{Q}$ and this set is denumerable. A denumerable union 
 of denumerable sets cannot exhaust the set of reals and the result follows. 
\end{proof}

This proves that the set $f(\mathbf{R})$ is dense in $\left(S^{1} \right)^{N}$.
Therefore, 
for every $\delta>0$ there exists $t_0 \in \mathbf{R}$ such that 
$ |f(t_0)-A^0| \leq \delta$, and hence, since $F$ is continuous, 
\begin{equation} \label{bb2}
|B_n^{\leq N}(t_0) -M|=|F(f(t_0))-F(A^0)|< \epsilon.
\end{equation}
Since $M$ is the maximum, for all $t\in \mathbf{R}$, we have
\begin{equation} \label{bb3}
 |B_n^{\leq N}(\sigma_1+it)| \leq M\leq |B_n^{\leq N}(\sigma_1+it_0)|+|B_n^{\leq N}(\sigma_1+it_0) -M|.
\end{equation}
By equations (\ref{bb1}), (\ref{bb2})  and (\ref{bb3}) we now have
$$ |B_n(\sigma_1 + i t)| \leq  \epsilon + |B_n^{\leq N}(\sigma_1 + it)| \leq \epsilon + |B_n^{\leq N}(\sigma_1+it_0)|+|B_n^{\leq N}(\sigma_1+it_0) -M| \leq 2 \epsilon +  |B_n^{\leq N}(\sigma_1+it_0)|,
$$
and this finishes the proof of lemma \ref{bound-lemma}.

 \end{proof} The proof of the lemma also completes the proof of the theorem. 
 \end{proof}


\begin{app} Let $X$ denote a closed smooth Riemannian manifold and let $a \in C^{\infty}(X)$ denote a smooth function. Define a generalized Dirichlet series by
$$ \zeta_{X,a}:=\mathrm{tr}(a\Delta_X^{-s}), $$
cf.\ \cite{CdJ}. Then 
$$ \zeta_{X,a} = \sum_{0 \neq \lambda \in \Lambda_X} \frac{1}{\lambda^s} \cdot \int_X a \sigma_{X,\lambda}, $$
where 
$$ \sigma_{X,\lambda}:=\sum_{\lambda \dashv \Psi} |\Psi|^2 $$
is the sum of the elements $\Psi$ of an orthonormal basis of eigenfunctions that belong to the eigenvalue $\lambda$. In \cite{CdJ},  it was proven that a diffeomorphism $\varphi \colon Y \rightarrow X$ between closed Riemannian manifolds \emph{with simple spectrum} is an isometry precisely if $\zeta_{X,a} = \zeta_{Y,\varphi^*(a)}$ for all $a \in C^\infty(X)$ (and there is also a version if the spectrum is not simple). 

Now assume that we have a compact manifold $X$ and a family $\{g_r\}$ ($r \in \R$) of isospectral metrics with simple eigenvalues on $X$ (cf.\ Gordon and Wilson \cite{GW} for the existence of such families). Denote by $\Psi_{r,\lambda}$ the normalized real eigenfunction for the metric $g_r$ corresponding to the eigenvalue $\lambda$. If all zeta functions converge in the sense that  
\begin{equation} \label{ZZ} \zeta_{X,g_r,a} \rightarrow \zeta_{X,g_s,a} \mbox{ for all } a \end{equation}
then we find from the above result that $$\int a \Psi^2_{r,\lambda} d\mu_r \rightarrow \int a \Psi^2_{s,\lambda} d\mu_s $$ for all functions $a \in C^\infty(X)$, where $\mu_r$ is the measure belonging to the metric $g_r$. 

Taking residues at $\dim(X)/2$ in (\ref{ZZ}) for $a=1$, we find that the volume of $X$ in $g_r$ is constant, and then taking residues for general $a$, we find that for all $a \in C^\infty(X)$, 
$$ \int a d\mu_r  \rightarrow \int a d\mu_s. $$
Changing variables, we get that 
$$ \int a \left(1-\frac{d\mu_r}{d\mu_s}\right) d\mu_s \rightarrow 0, $$
and hence that $\frac{d\mu_r}{d\mu_s} \rightarrow 1$. 
From the above theorem, we conclude that $\Psi^2_{r,\lambda}\frac{d\mu_r}{d\mu_s} \rightarrow \Psi^2_{s,\lambda},$  and hence that 
 $$\Psi^2_{r,\lambda} \rightarrow \Psi^2_{s,\lambda},$$
 a convergence of squared eigenfunctions. 

\end{app}

%

\section{General case} \label{general} 

Finally, in the most general case of varying coefficients and varying exponents, we prove a theorem about accumulation points. First, we do some preparation.

\begin{definition} \label{strongI} For a fixed strictly positive real function $g$, define for a real function $f$, the $g$-sup norm as $$||f||_{\infty,g}:=\sup_{x \in \R} \left|\frac{f(x)}{g(x)}\right|,$$ when it is defined. 
We say that a sequence of functions $\{f_n\}$ \emph{converges multiplicatively}
 to a real function $f$ if there exists a strictly positive real function $g$ that is integrable with respect to the multiplicative Haar measure on $\R^*$ (i.e., such that $\int_{\R} g(x)\frac{dx}{|x|}<+\infty$), such that 
$$|| f_n(x) - f(x) ||_{\infty,g} \rightarrow 0$$
for $n \rightarrow + \infty$. 
\end{definition}

\begin{definition} For $f$ a complex function defined for $\Re(s)=c$,  and $x \in \R$, 
denote by 
$$I^c_x(f):= \int_{\Re(s)=c} f(s) e^{xs} \frac{ds}{s}. $$
\end{definition}
The relevance of this integral for the theory of Dirichlet series lies in the following formula of Perron: if $D(s)=\sum_{\nu \geq 1} a_\nu e^{ -s \mu_\nu}$ is convergent for $s=\beta+i \gamma$ and 
$c>0$, $c > \beta$, $x\in \mathbf{R}$, $x \geq \beta$, then 
\[
\sum_{\lambda_n \leq x} a_n =\frac{1}{2 \pi i} I^c_x(D), 
\] with the convention that the last summand on the left hand side is multiplied by $1/2$ if $x$ equals some $\lambda_\nu$.

Since $I^c_x(D)$ does not depend on $c$ once it satisfies the conditions for Perron's formula, we will now write $I_x(D)$ for $I^c_x(D)$ with any suitable $c$. 
\begin{lemma} \label{inter}
If $\{f_n(c+it)\}$ converges multiplicatively to $f(c+it)$ in $t$, then for all $x \in \R$, 
\[
 \lim_n I_x(f_n)= I_x( \lim_n f_n)=I_x(f).
\]
\end{lemma}
\begin{proof}
 We have
\begin{align*}
\left| I_x(f_n)-I_x(f)\right| &\leq \int_{\R} \left| \frac{f_n(c+it)-f(c+it)}{ 
 c+it} e^{x(c+it)} \right| dt  \\
&\leq \int_{\R} \left| \frac{f_n(c+it)-f(c+it)}{c+it} e^{x(c+it)} \right| dt \\
 &\leq e^{cx} \left( \int_{\R} \frac{g(t)}{\sqrt{c^2+t^2}}dt \right) \cdot ||f_n(c+it)-f(c+it)||_{\infty,g} \\
  &\leq e^{cx} \left( \int_{\R} \frac{g(t)}{|t|}dt \right) \cdot ||f_n(c+it)-f(c+it)||_{\infty,g} \\
   & \leq C \varepsilon,
  \end{align*}
 with $C=e^{cx} \left( \int_{\R} \frac{g(t)}{|t|}dt \right)$ finite constant, 
 for $n$ sufficiently large. 
This proves the desired result.
\end{proof}

%
%
%
%

Before stating the main result of this section, we need to introduce some notation: 
\begin{note} \label{subseq}
Assume that all sequences $(\lambda_{n,j})_{n=1}^\infty$ are bounded.
 Let $ \ell_i^{(j)}$, $i\in I_j$ be the accumulation points of sequence $(\lambda_{n,j})_{n=1}^\infty$. 

We consider a subsequence $n_k$ such that for all $j$ 
$ \lim\limits_{n_k\rightarrow \infty }\lambda_{n_k,j}=\ell_{i_j}^{(j)}$ for a 
selection $i_j \in I_j$. Notice that the sequences $(\lambda_{n_k,j})_{k=1}^\infty$ and 
$(\lambda_{n_k,j+1})_{k=1}^\infty$ satisfy $\lambda_{n_k,j} < \lambda_{n_k,j+1}$ but they can tend to the same 
accumulation point. 

For the infinite vector  of convergent sequences $\big( (\lambda_{n_k,j})_{k=1}^\infty \big)_{j\geq 1}$ converging 
to the infinite vector $\big( \ell_{i_j}^{(j)} \big)_{j\geq 1}$ we consider the sequence $m_1,m_2,\ldots,$
such that 
\[\ell_{i_1}^{(1)}=\ell_{i_2}^{(2)}=\cdots=\ell_{i_{m_1}}^{(m_1)}, 
\ell_{i_{m_1+1}}^{(m_1+1)}=\ell_{i_{m_1+2}}^{(m_1+2)}=\cdots=\ell_{i_{m_2}}^{(m_2)}, \mbox{ etc.}
\]
\end{note}
\[
 \xymatrix{
(\lambda_{n_k,1}) \ar@{.}[r] \ar[d]  & (\lambda_{n_k,m_1})  \ar[dl] &  
(\lambda_{n_k,m_1+1}) \ar@{.}[r] \ar[d]  & (\lambda_{n_k,m_2})  \ar[dl]\\ 
 \ell_{i_1}^{(1)} &   &    \ell_{i_{m_1+1}}^{(m_1+1)} & 
} \cdots 
\]

\begin{theorem} \label{dif} 
We use the notation of \ref{subseq}.
Assume  that $D_n$ converges multiplicatively to $D(s)=\sum_{j \geq 1} a_j e^{-s \log \lambda_j}$.
Then, $\lambda_j$ are accumulation points for some sequence $(\lambda_{n,j'})_{n=1}^\infty$.

Consider the set of subsequences $\big( (\lambda_{n,j})_{n=1}^\infty \big)_{j\geq 1}$ converging 
to the infinite vector $\big( \ell_{i_j}^{(j)} \big)_{j\geq 1}$. 
Suppose that  the sequences $(\lambda_{n_k,j})_{n=1}^\infty$ for $j=m_\mu+1,\dots,m_{\mu+1}$ converge to $\ell$.
Set 
\[
A_{n_k}^{(\mu)}:=\sum_{j=m_\mu+1}^{m_{\mu+1}} a_{n_k,j}, \mbox{ for }  \mu \geq 0.
\]
Then, 
\begin{equation} \label{sum-fo}
 \lim A_{n_k}^{(\mu)}=\left\{
\begin{array}{l}
 a_i \ \ \mbox{if } \ell=\lambda_i, \\
 0 \ \  \mbox{ otherwise. }\\

\end{array}
\right.
\end{equation}
\end{theorem}
\begin{proof}
Assume that the set of subsequences  
$\big( (\lambda_{n_k,j})_{k=1}^\infty \big)_{j\geq 1}$ converges to the set of 
accumulation points $(\ell_i^{(j)})$. 

Consider the first eigenvalue $\lambda_1$ of $D$. If $\ell$ is the first  element in the set $\ell_i^{(j)}$ 
that is smaller than $\lambda_1$ then by choosing $x$ such that $\ell < x < \lambda_1$, by Perron's formula, we have that 
$I_x(D_{n_k})=\sum_{j=1}^{m_1} a_{n_k,j}=A_{n_k}^{(0)}$ should tend to $I_x(D)=0$ since $x<\lambda_1$. 
This proves that $A_{n_k}^{(0)}$ tends to zero as desired. We proceed now to the next accumulation point that 
is smaller than $\lambda_1$ and by the same argument we prove that 
$\lim\limits_{n_k\rightarrow \infty} \sum_{j=1}^{m_2} a_{n_k,j}=0$. Then, since the limit of the sum of the 
first $m_1$ terms tends to zero we have that 
\[
 \lim\limits_{n_k \rightarrow \infty} \sum_{j=m_1+1}^{m_2}a_{n_k,j}=0,
\]
and so the desired result is proved for all $\ell< \lambda_1$. 

We will prove now that $\lambda_1$ is an accumulation point. Indeed, for sufficiently small  $\epsilon>0$
the quantity 
\[
I_{\lambda_1-\epsilon}(D)-I_{\lambda_1+\epsilon}(D)=a_1 \neq 0.
\]
Using the above equation and  lemma \ref{inter} we obtain
\[
 \lim_{n_k \rightarrow \infty}  \big(I_{\lambda_1-\epsilon}(D_{n_k})-I_{\lambda_1+\epsilon}(D_{n_k}) \big)=a_1= 
\lim \sum_{\lambda_1-\epsilon < \lambda_{n_k,j} < \lambda_1+\epsilon} {a_{n_k,j}}.
\]
So by taking small $\epsilon$ we can find a  
subsequence tending to $\lambda_1$ so  $\lambda_1$ is one of the 
accumulation points of the sequence 
$\big( (\lambda_{n_k,j})_{k=1}^\infty \big)_{j\geq 1}$.
Notice also that desired  result of eq. (\ref{sum-fo}) is also 
proved.

We continue the proof by induction by taking $\ell$ to be inside $\lambda_1$ and $\lambda_2$ so the 
corresponding sum tends to zero, then we take $\ell$ to be $\lambda_2$, then inside $\lambda_2$
and $\lambda_3$ etc.

%
%
\end{proof}

\section{Relation with Laplace-Stieltjes Transform} \label{L}

The notion of Dirichlet series and Laplace transforms can be unified in 
terms of the Riemann-Stieltjes integrals (Widder \cite{Widder:29}, compare \cite{Arendt:2001}). 
%

\begin{definition} Suppose $\omega \geq 0$ is a real number. 
\begin{enumerate} \item The space $\mathrm{Lip}_\omega$ is defined as the set of functions $F \colon \R_{ \geq 0} \rightarrow \R$ with bounded norm 
$$ ||F||_{\mathrm{Lip},\omega}:= \sup_{0 \leq s < t}\frac{ |F(t)-F(s) |}{(t-s)e^{\omega t}} < \infty. $$
\item The space $\mathrm{Wid}_\omega$ is defined as the space of smooth function $(\omega,\infty) \rightarrow \R$ with bounded norm 
$$ ||D||_{\mathrm{Wid},\omega}:= \sup_{{s>\omega}\atop {k \in \mathbf{N}}} \frac{(s-\omega)^{k+1}}{k!} \left| \frac{d^k D}{ds^k}(s) \right| < \infty. $$
\end{enumerate}
\end{definition}
The main result is now that  the so-called \emph{Laplace-Stieltjes  transform} $$ F \mapsto  \int_0^\infty e^{-st} dF(t) $$  induces an isometric isomorphism 
$
\mathrm{Lip}_\omega \rightarrow \mathrm{Wid}_\omega
$ (\cite{Arendt:2001}, Thm.\ 2.4.1).

Widder (\cite{Widder:29}, Theorems 11.2 \and 12.4) proved that a Dirichlet series of the form $D(s)=\sum_{\nu} a_\nu e^{-s \mu_\nu}$ convergent for $\Re(s)>\omega$ is in the space $\mathrm{Wid}_{\omega}$. Also, such $D$ is the Laplace-Stieltjes transform of 
\[
F(t)=\sum_{\nu=0}^\infty a_\nu H(t-\mu_\nu),
\] where $H$ is the Heaviside step function. Thus, we immediately conclude the following: 
\begin{theorem} \label{final}
Suppose $D_n(s)=\sum_{\nu \geq 1} a_{n,\nu} e^{ -s \mu_{n,\nu}}$ is a sequence of Dirichlet series 
each converging absolutely in a common half plane $\Re(s)>\gamma$; then for any $\omega>\gamma$, $D_n$ converges to a Dirichlet series
$D(s)=\sum_{\nu \geq 1} a_\nu e^{ -s \mu_\nu}$ in $\mathrm{Wid}_{\omega}$-norm if and only if 
$$ \sum_{\nu=0}^\infty \big( a_{n,\nu} H(t-\mu_{n,\nu}) - a_\nu H(t-\mu_\nu) \big) \rightarrow 0$$ in $\mathrm{Lip}_\omega$-norm. \qed
\end{theorem}
It would be interesting to deduce Theorem \ref{qwe} and \ref{dif} from Theorem \ref{final}.

%
%
%

\bibliographystyle{plain}
\bibliography{spectral.bib}
\end{document}